\documentclass[a4,12pt]{article}%

\usepackage{graphicx}
\usepackage{graphics}

\usepackage{amsmath}
\usepackage{amsfonts}
\usepackage{amssymb}
\usepackage{enumerate}
\newtheorem{theorem}{Theorem}[section]

\newtheorem{lemma}[theorem]{Lemma}

\newenvironment{proof}[1][Proof]{\textbf{#1.} }{\ \rule{0.5em}{0.5em}}

\newcommand{\ind}{{\rm fpi}}
\newcommand{\IND}{{\rm FPI}}

\newcommand{\dR}{{\bf R}}

\newcommand{\calE}{{\cal E}}

\newcounter{figurecounter}
\begin{document}

\title{Browder's Theorem with General Parameter Space%
\thanks{We thank John Yehuda Levy and Orin Munk for discussions that inspired this work.
The first author acknowledges the support of the Israel Science Foundation, Grant \#217/17.}}

\author{Eilon Solan and Omri N.~Solan%
\thanks{The School of Mathematical Sciences, Tel Aviv
University, Tel Aviv 6997800, Israel. e-mail: eilons@post.tau.ac.il, omrisola@post.tau.ac.il.}}

\maketitle

\begin{abstract}
Browder (1960) proved that for every continuous function $F : X \times Y \to Y$,
where $X$ is the unit interval and $Y$ is a nonempty, convex, and compact subset of $\dR^n$,
the set of fixed points of $F$, defined by $C_F := \{ (x,y) \in X \times Y \colon F(x,y)=y\}$
has a connected component whose projection to the first coordinate is $X$.
We extend this result to the case where $X$ is a connected and compact Hausdorff space.
\end{abstract}


\bigskip


\section{Introduction}

Brouwer's Fixed Point Theorem states that every continuous function from a nonempty, convex, and compact subset of a Euclidean space into itself
has a fixed point.
Browder (1960) proved the following parametric extension of Brouwer's Fixed Point Theorem.

\begin{theorem}[Broweder, 1960]
\label{theorem:browder}
Let $X = [0,1]$, let $Y \subseteq \dR^n$ be nonempty, convex, and compact,
and let $F : X \times Y \to Y$ be a continuous function.
Define
the set of fixed points of $F$ by
\begin{equation}
\label{equ:98}
C_F := \{ (x,y) \in X \times Y \colon F(x,y)=y\}.
\end{equation}
Then $C_F$ has a connected component whose projection to the first coordinate is $X$.
\end{theorem}

Theorem~\ref{theorem:browder} was used in a variety of topics,
like nonlinear complementarity theory (see, e.g., Eaves, 1971, or Allgower and Georg, 2012),
nonlinear elliptic boundary value problems
(Shaw, 1977),
the study of global continua of solutions of nonlinear
partial differential equations (see, e.g., Costa and Gon\c{c}alves, 1981, or Massabo and Pejsachowitz, 1984),
theoretical economics (Citanna et al., 2001),
and game theory (see, e.g., Herings and Peeters, 2010, or Solan and Solan, 2021).

A natural question is whether Browder's Theorem extends to the case that $X$ is not the unit interval,
but rather any nonempty, connected, and compact Hausdorff space.
Such an extension was in fact required by Munk and Solan (2020) in their study in game theory.

When $X$ admits a space-filling curve,
then the extension of Browder's Theorem follows easily from Theorem~\ref{theorem:browder}.
By the Hahn-Mazurkiewicz Theorem (see, e.g., Willard, 2012, Theorem 31.5),
a non-empty Hausdorff topological space admits a space-filling curve if and only if it is compact, connected, locally connected, and metrizable.
One example of a compact, connected, and metrizable space that
is neither locally connected nor does it admit a space-filling curve is
\[ V := \{ (x,y) \in [0,1] \times [-1,1] \colon x=0, \hbox{ or } x > 0 \hbox{ and } y = \sin(1/x)\}. \]

In this note we extend Browder's Theorem to the case where
$X$ is a connected and compact Hausdorff space,
and $Y$ is a nonempty, convex, and compact subset of a finite-dimensional real vector space.
Another type of extension of Browder's Theorem to the case
where $F$ is an upper hemi-continuous set-valued function with contractible values
(yet $X$ is still the unit interval) was provided by Mas-Colell (1974).

In addition to extending Browder's Theorem to a more general setup,
our proof ensures that the topological properties of the
connected component of $C_F$ whose projection to the first coordinate is $X$
are the same as those of $X$:
both are connected and compact Hausdorff spaces.
This allows to apply Browder's Theorem to the connected component whose existence is guaranteed by the theorem.
That is,
suppose that $X$ is a compact, connected, locally connected, and metrizable Hausdorff space, and $Y$ and $F$ are as in Theorem~\ref{theorem:browder}.
Let $D$ be a connected component of $C_F$ whose projection to the first coordinate is $X$,
and whose existence is guaranteed by Theorem~\ref{theorem:browder}.
Suppose further that $G : D \times Y' \to Y'$ is a continuous function,
where $Y'$ is a finite-dimensional real vector space.
Since $D$ is not necessarily locally connected,
to apply Browder's Theorem to $G$ we need a version of Theorem~\ref{theorem:browder} that holds
for spaces that are not necessarily locally connected.

\section{The Main Result}

The main result of this note is the following extension of Theorem~\ref{theorem:browder}.

\begin{theorem}
\label{theorem:browder1}
Let $X$ be a connected and compact Hausdorff space,
let $Y$ be nonempty, compact, and convex set in a finite-dimensional real vector space $V$,
and let $F : X \times Y \to Y$ be a continuous function.
Then $C_F$, the set of fixed point of $F$ that is defined in Eq.~\eqref{equ:98},
 has a connected component whose projection over the first coordinate is $X$.
\end{theorem}

We note that since $X$ and $Y$ are compact,
and since $F$ is continuous,
the set $C_F$ is compact.

For the proof we will need the concept of the fixed-point index, which we describe in Section~\ref{section:index}.
The proof of Theorem~\ref{theorem:browder1} appears in Section~\ref{section:proof}.
Though Browder's (1960) original proof uses the fixed-point index as well,
it relies on the fact that the unit interval has two extreme points,
and hence it is not clear whether it can be easily extended to any connected and compact Hausdorff space.

\subsection{The Fixed Point Index}
\label{section:index}

Roughly,
the fixed-point index is a mapping that counts with signs the number of fixed points of a continuous function in a given set.
Given a finite-dimensional real vector space $V$,
the fixed-point index is a function, denoted $\ind$, that assigns an integer to every
pair $(Z,f)$, where $Z \subseteq V$ is compact
and $f : Z \to V$ is a continuous function that has no fixed points on $\partial Z$, the boundary of $Z$.
The function $\ind$ satisfies the following properties:
\begin{itemize}
\item Normalization:
If $f$ is a constant function whose value is an element of the interior of $Z$,
then $\ind(Z,f) = 1$.
\item Additivity:
For every compact sets $Z_1,\dots,Z_K$ with disjoint interiors
and every continuous function $f : \bigcup_{k=1}^K Z_k \to V$ such that $f$ has no fixed points on $\bigcup_{k=1}^K \partial Z_k$,
we have
\[ \ind\left(\bigcup_{k=1}^K Z_k,f\right) = \sum_{k=1}^K \ind(Z_k,f). \]
\item   Continuity:
For every compact set $Z$,
the function $f \mapsto \ind(Z,f)$ is continuous over the set of continuous functions $f : Z \to V$ that do not have fixed points in $\partial Z$.
\end{itemize}

It follows from the additivity and continuity properties that if $f$ has no fixed points in $Z$, then $\ind(Z,f) = 0$.

For an exposition of the fixed-point index,
as well as proofs of the existence and uniqueness of the fixed-point index,
see, e.g., McLennan (2018, Chapter~13).

We will present a reformulation of the fixed-point index that is adapted for our purposes.
For every topological space $C$ denote by $\calE(C)$ the collection of all clopen subsets of $C$,
namely, all subsets that are both closed and open.
For every subset $Y$ of a finite-dimensional real vector space
and every function $f : Y \to Y$,
denote the set of fixed point of $f$ by
\[ C_f := \{ y \in Y \colon f(y) = y\} \subseteq Y. \]
For every two sets $X$ and $Y$, every function $F : X \times Y \to Y$,
and every $x \in X$, denote by $F_x : Y \to Y$ the function defined by
\[ F_x(y) := F(x,y), \ \ \ \forall y \in Y. \]


\begin{theorem}
Let $X$ be a connected and compact Hausdorff space,
and let $Y$ be nonempty, compact, and convex set in a finite-dimensional real vector space.
There exists (a) a function $\ind_*$ that assigns an integer
to each pair $(\widehat D,f)$, where $f : Y \to Y$ is continuous and $\widehat D \in \calE(C_f)$,
and
(b) a function $\IND_*$ that assigns an integer to each pair $(D,F)$, where $F : X \times Y \to Y$ is continuous and $D \in \calE(C_F)$,
such that the following properties hold
for every continuous functions $f : Y \to Y$ and $F : X \times Y \to Y$:
\begin{enumerate}
\item[I1)]   Normalization: $\ind_*(C_f,f) = 1$ and $\IND_*(C_F,F) = 1$.
\item[I2)]   Additivity: for every disjoint sets $(\widehat D_k)_{k=1}^K$ in $\calE(C_f)$,
\[ \ind_*\left(\bigcup_{k=1}^K \widehat D_k,f\right) = \sum_{k=1}^K \ind_*(\widehat D_k,f), \]
and for every disjoint sets $(D_k)_{k=1}^K$ in $\calE(C_F)$,
\[ \IND_*\left(\bigcup_{k=1}^K D_k,F\right) = \sum_{k=1}^K \IND_*(D_k,F). \]
\item[I3)]  Compatibility: for every compact set $Z$ of $Y$ such that $\partial Z \cap C_f = \emptyset$
\[ \ind_*(Z \cap C_f,f) = \ind(Z,f). \]
\item[I4)]   Locality: For every $E \in \calE(C_F)$,
\[ \IND_*(E,F) = \ind_*\bigl(E \cap (\{x\} \times Y),F_x\bigr), \ \ \ \forall x \in X. \]
\end{enumerate}
\end{theorem}

Note that while the fixed-point index is defined for pairs $(\widehat D,f)$ such that $f$ has no fixed point on $\partial \widehat D$,
the function $\ind_*$ is defined for pairs $(\widehat D,f)$ where all points in $\widehat D$ are fixed points of $f$.

\bigskip

\begin{proof}
We first argue that we can assume w.l.o.g.~that $f$ has no fixed points on $\partial Y$.
Indeed, let $\widehat Y \subset V$ be convex and compact set whose interior contains $Y$.
Given a function $f : Y \to Y$, define a function $\widehat f : \widehat Y \to Y$ by
\[ \widehat f(y) := f(\pi_Y(y)), \ \ \ \forall y \in Y, \]
where $\pi_Y : \widehat Y \to Y$ is the projection.
For every continuous function $f : Y \to Y$,
the function $\widehat f$ is continuous and has no fixed points in $\widehat Y \setminus Y$.
We can then study the fixed-point index when the underlying space is $\widehat Y$ rather than $Y$.

For every clopen set $D \subseteq C_f$, let $Z_D \subseteq \widehat Y$ be a compact set that satisfies
(a) $Z_D \supset D$,
(b) $\partial Z_D \cap \partial D = \emptyset$
and
(c) $Z_D \cap C_f = D$.
Such a set exists since $D$ is clopen and $C_f$ is compact.
It follows that $\ind(Z_D, f)$ is well defined.
Set
\[ \ind_*(D,f) := \ind(Z_D, f). \]
Standard arguments show that this definition is independent of the choice of $Z_D$,
and that (I1)--(I4) hold by the properties of the fixed-point index.
\end{proof}

\bigskip

One consequence of the definition is that if $D \in \calE(C_F)$, and if $D \cap (\{x\} \times Y) = \emptyset$,
then $\IND_*(D,F) = 0$.
Indeed, by (I4) we have $\IND_*(D,F) = \ind_*(\emptyset,F_x)$,
and by (I2) the latter is equal to 0.

\subsection{The Proof of Theorem~\ref{theorem:browder1}}
\label{section:proof}




Theorem~\ref{theorem:browder1} follows from the following two lemmas.

\begin{lemma}
\label{lemma:5}
There exists a connected component $D \subseteq C_F$ that satisfies the following property:
for every $E \in \calE(C_F)$ that contains $D$
there exists $E' \in \calE(C_F)$
that satisfies $D \subseteq E' \subseteq E$ and $\IND_*(E',F) \neq 0$.
\end{lemma}

\begin{lemma}
\label{lemma:6}
Let $D \subseteq C_F$ be a connected component that satisfies the condition of Lemma~\ref{lemma:5}.
The projection of $D$ to the first coordinate is $X$.
\end{lemma}

\begin{proof}[Proof of Lemma~\ref{lemma:5}]
Assume to the contrary that the claim does not hold.
Then for every connected component $D$ of $C_F$ there exists $E_D \in \calE(C_F)$ such that
$D \subseteq E_D$ and for every $E' \in \calE(C_F)$ that satisfies $D \subseteq E' \subseteq E_D$ we have $\IND_*(E',F) = 0$.
It follows that $\IND(E_D,F) = 0$, and hence
 $\IND(\widehat E,F) = 0$ for every clopen subset $\widehat E$ of $E_D$,
whether or not $\widehat E$ contains $D$.
Indeed,
if $\widehat E$ does not contain $D$.
then it is disjoint of $D$, hence
by (I2) we have
\[ \IND_*(E_D \setminus \widehat E,F) = \IND_*(E_D,F) - \IND_*(\widehat E,F) = 0 - 0 = 0. \]

Since $C_F$ is compact and $(E_D)_D$ is an open cover of $C_F$,
there is a finite collection of its connected components, denoted $(D_k)_{k=1}^K$,
such that $(E_{D_k})_{k=1}$ covers $C_F$.

The set $C_F$ can be presented as a disjoint union
\[ C_F = \bigsqcup_{k=1}^K \left( E_{D_k} \setminus \left(\bigcup_{j < k} E_{D_j}\right)\right). \]
Since the sets $(E_{D_k})_{k=1}^K$ are clopen, so is the set $E_{D_k} \setminus \left(\bigcup_{j < k} E_{D_j}\right)$, for each $k$, $1 \leq k \leq K$.
Since $E_{D_k} \setminus \left(\bigcup_{j < k} E_{D_j}\right) \subseteq E_{D_k}$,
it follows from the definition of $E_{D_k}$ that
$\ind_*\left(E_{D_k} \setminus \left(\bigcup_{j < k} E_{D_j}\right),F\right) = 0$.
By (I1) and (I2) we have
\[ 1 = \IND_*(C_F,F)
= \sum_{k=1}^K \IND_*\left( E_{D_k} \setminus \left(\bigcup_{j < k} E_{D_j}\right), F \right) = 0,
\]
a contradiction.
\end{proof}

\bigskip

\begin{proof}[Proof of Lemma~\ref{lemma:6}]
Assume to the contrary that the claim does not hold.
Then there exists $x \in X$ that is not in the projection of $D$ to the first coordinate.
It follows that $(\{x\} \times Y) \cap D = \emptyset$.

The set $D$ is the intersection of all clopen sets that contain it
(see, e.g., Kuratowski, 1968, Theorem 2 on Page 169).
Denoting by $(E_\alpha)_\alpha$ all the clopen sets in $\calE(C_F)$ that contain $D$,
we have
\[ (\{x\} \times Y) \cap \left(\bigcap_\alpha E_\alpha\right) = \emptyset. \]
Since all the sets in this intersection are compact,
 a finite intersection of these sets is already empty:
 there are $(\alpha_k)_{k=1}^K$ such that
\begin{equation}
\label{equ:2}
(\{x\} \times Y) \cap \left(\bigcap_{k=1}^K E_{\alpha_k}\right) = \emptyset.
\end{equation}
Set $D^* := \bigcap_{k=1}^K E_{\alpha_k}$.
As the intersection of finitely many clopen sets, $D^* \in \calE(C_F)$.
Eq.~\eqref{equ:2} implies that $IND_*(D',F) = 0$,
for every $D' \subseteq D^*$ such that $D' \in \calE(C_F)$.
This contradicts the choice of $D$.
Indeed, we selected $D$ to satisfy Lemma~\ref{lemma:5},
yet the set $E = D^*$ does not satisfy the conclusion of the lemma.
\end{proof}


\begin{thebibliography}{ZZ}
\bibitem{AG}
Allgower E.L. and Georg K. (2012)
\emph{Numerical Continuation Methods: an Introduction}, Springer Science \& Business Media.

\bibitem{Browder}
Browder F. (1960)
On Continuity of Fixed Points under Deformation of Continuous Mappings.
\emph{Summa Brasiliensis Mathematicae}, 4, 183--191.

\bibitem{Citanna}
Citanna A., Cr\`es H., Drèze J., Herings P.J.J., and Villanacci A. (2001)
Continua of Underemployment Equilibria Reflecting Coordination Failures, Also at Walrasian Prices,
\emph{Journal of Mathematical Economics}, 36, 169--200.

\bibitem{CG}
Costa D.G. and Gon\c{c}alves J.V.A. (1981)
Existence and Multiplicity Results for a Class of Nonlinear Elliptic Boundary Value Problems at Resonance,
\emph{Journal of Mathematical Analysis and Applications}, 84(2), 328--337.

\bibitem{Eaves}
Eaves B.C. (1971) On the Basic Theorem of Complementarity,
\emph{Mathematical Programming}, 1(1), 68--75.

\bibitem{HP}
Herings P.J.J. and Peeters R. (2010)
Homotopy Methods to Compute Equilibria in Game Theory,
\emph{Economic Theory}, 42(1), 119--156.

\bibitem{Kuratowski}
Kuratowski K. (1968)
\emph{Topology}, Volume II, Academic Press.

\bibitem{MasCollel}
Mas-Colell A. (1974)
A Note on a Theorem of F. Browder,
\emph{Mathematical Programming}, 6(1), 229--233.

\bibitem{McLennan}
McLennan A. (2018)
\emph{Advanced Fixed Point Theory for Economics}. Singapore: Springer.

\bibitem{MP}
Massabo I. and Pejsachowicz J. (1984)
On the Connectivity Properties of the Solution Set of Parametrized Families of Compact Vector Fields,
\emph{Journal of Functional Analysis}, 59(2), 151--166.

\bibitem{MunkSolan}
Munk O. and Solan E. (2020)
Sunspot Equilibrium in Positive Recursive Two-Dimensions Quitting Absorbing Games.
arXiv preprint arXiv:2001.03094.

\bibitem{Shaw}
Shaw H. (1977)
A Nonlinear Elliptic Boundary Value Problem at Resonance,
\emph{Journal of Differential Equations}, 26, 335--346.

\bibitem{SS2021}
Solan E. and Solan O.N. (2021)
Sunspot Equilibrium in Positive Recursive General Quitting Games,
\emph{International Journal of Game Theory}, forthcoming.

\bibitem{Willard 2012}
Willard S. (2012) \emph{General Topology}, Courier Corporation.
\end{thebibliography}
\end{document}